\documentclass{article}
 \pdfoutput=1 
\usepackage[utf8]{inputenc}
\usepackage{fullpage}
\usepackage{amsmath,amsfonts,amssymb,amsthm,enumerate,mdframed,hyperref,cleveref,algorithmic,algorithm,multirow,import,tikz,booktabs,xcolor,soul}
\usetikzlibrary{arrows,positioning, shapes} 
\usepackage{caption,subcaption}
\usepackage{authblk}

\newtheorem{theorem}{Theorem}
\newtheorem{proposition}[theorem]{Proposition}
\newtheorem{lemma}[theorem]{Lemma}

\newtheorem{example}[theorem]{Example}

\mathcode`*=\numexpr\mathcode`*-"2000\relax

\definecolor{darkgreen}{rgb}{0,0.6,0}

\providecommand{\F}{\mathbb{F}}
\providecommand{\N}{\mathbb{N}}
\providecommand{\Z}{\mathbb{Z}}

\DeclareMathOperator{\PG}{PG}

\newcommand{\cM}{\mathcal{M}}

\newcommand{\cS}{\mathcal{S}}

\title{Lengths of divisible codes -- the missing cases}
\author{Sascha Kurz}
\date{}
\affil{Mathematisches Institut, Universit\"at Bayreuth, D-95440 Bayreuth, Germany, sascha.kurz@uni-bayreuth.de}

\begin{document}

\maketitle

\begin{abstract}
  A linear code $C$ over $\mathbb{F}_q$ is called $\Delta$-divisible if the Hamming weights $\operatorname{wt}(c)$ of all codewords $c \in C$ are 
  divisible by $\Delta$. The possible effective lengths of $q^r$-divisible codes have been completely characterized for each prime power $q$ and 
  each non-negative integer $r$ in \cite{kiermaier2020lengths}. The study of $\Delta$-divisible codes was initiated by Harold Ward \cite{ward1981divisible}. 
  If $t$ divides $\Delta$ but is coprime to $q$, then each $\Delta$-divisible code $C$ over $\F_q$ is the $t$-fold repetition of a $\Delta/t$-divisible 
  code. Here we determine the possible effective lengths of $p^r$-divisible codes over finite fields of characteristic $p$, where $r\in\mathbb{N}$ but $p^r$ is 
  not a power of the field size, i.e., the missing cases.
  
  \smallskip
  
\end{abstract}

\section{Introduction}

A linear code $C$ over $\mathbb{F}_q$ is called $\Delta$-divisible if the Hamming weights $\operatorname{wt}(c)$ of all codewords $c \in C$ are 
divisible by $\Delta$. The study of divisible codes was initiated by Harold Ward \cite{ward1981divisible}. Linear codes meeting the Griesmer bound 
in many cases have to admit a relatively large divisibility constant $\Delta$, see \cite{ward1998divisibility}. In order to state a      
connection between divisible codes and Galois geometries we associate each subspace $U\in\operatorname{PG}(v-1,q)$ with its characteristic function 
$\chi_U$ mapping each point $\operatorname{PG}(v-1,q)$ to a non-negative integer multiplicity, i.e., $\chi_U(P)=1$ iff $P\le U$ and $\chi_U(P)=0$ otherwise. 
We say that a mapping $\cM$ from the point set of $\operatorname{PG}(v-1,q)$ to $\mathbb{N}$ is $\Delta$-divisible if the corresponding linear 
code $C_{\cM}$ associated with the multiset of points characterized by $\cM$ is $\Delta$-divisible. We call $\cM(P)$ the multiplicity of a point $P$ and 
extend this notion additively to arbitrary subspaces $S$ by letting $\cM(S)$ be the sum over all point multiplicities $\cM(P)$ where $P$ is contained in $S$. 
If $S$ is the entire ambient space then we speak of the cardinality $\#\cM$. Using this notion, we can more directly state that a multiset of points $\cM$ in 
$\PG(v-1,q)$ is $\Delta$-divisible iff we have $\# \cM\equiv \cM(H)\pmod\Delta$ for every hyperplane $H$. The effective length of $C_{\cM}$ equals the cardinality 
$\#\cM$. We say that a multiset $M$ of points is spanning if the points with positive multiplicity span the entire ambient space. Using the geometric language 
we will call $1$-, $2$-, $3$-, and $(v-1)$-dimensional subspaces points, lines, planes, and hyperplanes, respectively.  

In e.g.\ \cite[Lemma 11]{kiermaier2020lengths} it was shown that 
for each multiset of subspaces $\mathcal{U}$ in $\operatorname{PG}(v-1,q)$ that have dimensions at least $k$ the multiset of points $\chi_{\mathcal{U}}:= 
\sum_{U\in\mathcal{U}} \chi_U$ is $q^{k-1}$-divisible. Since also the complementary multiset of points is $q^{k-1}$-divisible, non-existence results for 
$q^{k-1}$-divisible codes imply upper bounds on the maximum size of $k$-spreads, see e.g.\ \cite{honold2018partial} for more details. Similarly, non-existence 
results for so-called vector space partitions, i.e., set of subspaces partitioning the point set of $\PG(v-1,q)$, can be deduced from certain non-existence 
results for $q^r$-divisible codes, see e.g.\ \cite{kurz2022vector}. Also in the situation where some points can be contained in several subspaces non-existence results for $q^{k-1}$-divisible 
codes can be applied to deduce results for problems in Galois geometry. For constant-dimension codes, i.e., sets of $k$-dimensional subspaces of $\PG(v-1,q)$ 
such that the dimensions of the pairwise intersections are upper bounded by some integer, they can be utilized for upper bounds on the cardinality, see 
e.g.\ \cite[Theorem 12]{kiermaier2020lengths}. For similar bounds for mixed-dimension subspace codes, where the codewords can have different dimensions, we 
refer to \cite{honold2019johnson}. Two surveys on applications of divisible codes are given by \cite{kurz2021divisible,ward2001divisible}.


The possible (effective) lengths of $q^r$-divisible codes have been completely characterized for each prime power $q$ and each non-negative integer $r$ in 
\cite[Theorem 1]{kiermaier2020lengths}. An important structure result for $\Delta$-divisible codes $C$ over $\F_q$ was shown in \cite{ward1981divisible}: If $t\in \N$
divides $\Delta$ and is coprime to $\Delta$ then there exists a $\Delta/t$-divisible code $C'$ over $\F_q$ such that $C$ is the $t$-fold repetition of $C'$. 
So, it suffices to study the possible (effective) lengths of $p^e$-divisible codes over $\F_q$, where $p$ is the characteristic of the field and $e$ an integer. 
When $q$ is not a prime the characterization result from \cite{kiermaier2020lengths} does not give an answer for the cases when the divisibility constant 
$\Delta$ is not a power of the field size (but only its characteristic). Here we close this gap and state a corresponding characterization of the possible 
(effective) lengths in Theorem~\ref{main_thm}.   

A few general (and easy) constructions for $\Delta$-divisible multisets of points $\cM$ in $\PG(v-1,q)$ are known, see e.g.\ \cite{kiermaier2020lengths} for proofs:
\begin{itemize}
  \item[(1)] if a multiset of points $\cM$ in $\PG(v-1,q)$ is $\Delta$-divisible, then there exists an embedding $\cM'$ of $\cM$ in $\PG(v'-1,q)$ for each $v'\ge v$ that is also $\Delta$-divisible;
  \item[(2)] if multisets of points $\cM,\cM'$ in $\PG(v-1,q)$ are $\Delta$-divisible with cardinalities $n,n'$, then $\cM+\cM'$ is $\Delta$-divisible with cardinality $n+n'$ in $\PG(v-1,q)$;  
  \item[(3)] if a multiset of points $\cM$ in $\PG(v-1,q)$ is $\Delta$-divisible with cardinalities $n$, then $c\cdot \cM$ is $c\cdot\Delta$-divisible with cardinality $cn$ for each positive integer $c$;
  \item[(4)] for each integer $u\ge 1$ and each $u$-dimensional subspace $U$ in $\PG(v-1,q)$ the corresponding characteristic function $\chi_U$ is $q^{u-1}$-divisible with cardinality $\tfrac{q^u-1}{q-1}$. 
\end{itemize}  
Here we are only interested in the possible cardinalities of $\Delta$-divisible multisets of points and not in their dimensions, see (1). However, in some applications 
there are restrictions on the dimension, so that the determination of the minimum dimension of a divisible code of given length remains an interesting open problem. Assume $q=p^e$ for a 
prime $p$ and an integer $e$. With this we will use the parameterization $\Delta=p^{ae-b}$ where $a,b\in\N$ with $a\ge 1$ and $b\le e-1$. By (4) an $(a+1)$-dimensional subspace is $q^a$-divisible. 
Since $ae \ge ae-b$ it is also $p^{ae-b}$-divisible. For $1 \le i \le a$ we can consider an $p^{ie-b}$-fold $(a-i+1)$-dimensional subspace which is $p^{ae-b}$-divisible using (3). Let us 
denote the corresponding cardinalities by $s_q (a, b, i)$, where $1\le i\le a$, and write $s_q(a,b,0)$ for the cardinality of an $(a+1)$-dimensional subspace. Using (2) we conclude
that for each $c_0,\dots,c_a\in\mathbb{N}$ there exists a $p^{ae-b}$-divisible multiset of points of cardinality
$$
  n=\sum_{i=0}^a c_i \cdot s_q(a,b,i)
$$
in $\PG(v-1,q)$ for sufficiently large dimension $v$ of the ambient space. Our main theorem, see Theorem~\ref{main_thm}, will state that for other cardinalities there is no $p^{ae-b}$-divisible 
multiset of points and we will give a direct characterization of the attainable cardinalities, i.e., we solve the so-called Frobenius coin problem for the {\lq\lq}coin values{\rq\rq} $s_q(a,b,0),\dots, 
s_q(a,b,a)$. For a solution of the Frobenius coin problem for geometric sequences we refer to \cite{ong2008frobenius}. 

\medskip

The remaining part of this paper is structured as follows. In Section~\ref{sec_main_thm} we prove our main theorem and in Section~\ref{sec_projective} we consider the possible cardinalities of $\Delta$-divisible 
sets of points. In the latter section we can only state a few numerical results and leave the general problem widely open. We especially study $2$-divisible sets of points and obtain a few preliminary results. 
Related results can be found in the literature under the terms of sets of odd and of even type, see e.g.\ \cite{adriaensen2023note,hirschfeld1980sets,key1998small,limbupasiriporn2010small,sherman1983sets,
tanaka2013classification,weiner2014stability}. 

\section{The generalized theorem}
\label{sec_main_thm}
For each integer $i\ge 1$ we define $[i]_q:=\tfrac{q^i-1}{q-1}$, i.e., the number of points of an $i$-dimensional subspace. For each prime power $q$ we write $q=p^e$, where $p$ is the characteristic 
of $\F_q$. When we consider $\Delta$-divisible codes over $\F_q$ we assume that $\Delta$ is a power of $p$. More concretely, we will use the parameterization $\Delta=p^{ae-b}$ where 
$a,b\in\N$ with $a\ge 1$ and $b\le e-1$. Sometimes we will also use $f:=ae-b$, i.e., the exponent in $\Delta=p^f$. For a fixed prime power $q=p^e$, non-negative integers $a,b$ with $a\ge 1$,
$b\le e-1$, and $i\in\{0,\dots,a\}$ we define
\begin{equation}
  s_q(a,b,i):=[a+1]_q
\end{equation}
if $i=0$ and
\begin{equation}
  s_q(a,b,i):=q^i\cdot [a-i+1]_q/p^b= p^{ie-b}\cdot [a-i+1]_q=p^{e-b}\cdot\left(q^{i-1}+q^i+\dots+q^{a-1}\right)
\end{equation}
for $1\le i\le a$. Note that for $i\ge 1$ the number $s_q(a,b,i)$ is divisible by $p^{ie-b}$ but not by $p^{ie-b+1}$, where $ie-b\ge 1$, and $s_q(a,b,0)$ is coprime to $p$. This property allows 
us to create kind of a positional system upon the sequence of base numbers
$$
  S_q(a,b) := \big(s_q(a,b,0),s_q(a,b,1),\dots,s_q(a,b,a)\big).
$$  
As it can be easily shown, each integer $n$ has a unique \emph{$S_q(a,b)$-adic expansion}
\begin{equation}
  n =\sum_{i=0}^a c_i \cdot s_q(a,b,i)
\end{equation}
with $c_0 \in \left\{0,\dots,p^{e-b}-1\right\}$,  $c_1,\dots,c_a-1 \in \{0\dots,q-1\}$ and \emph{leading coefficient} $c_a\in \Z$. The sum $p^bc_0\,+\,c_1+c_2+\dots+c_a$ will be called the  
\emph{cross sum} of the $S_q(a,b)$-adic expansion of $n$.

\begin{example}
  For $q=8$ and $\Delta=32$ we have $p=2$, $e=3$, $a=2$, $b=1$, and 
  $$
    S_8(2,1) = \big(s_8(2,1,0),s_8(2,1,1),s_8(2,1,2)\big)=\big(73,36,32\big).
  $$
  The characteristic function $\chi_E$ of a plane in $\PG(v-1,8)$ is $64$-divisible with cardinality $s_8(2,1,0)=73$. Since the characteristic function $\chi_L$ of a line over $\F_8$ 
  is $8$-divisible $4\cdot\chi_L$ is $32$-divisible with cardinality $s_8(2,1,1)=36$. A $32$-fold point corresponds to a $32$-divisible multiset of points in $\PG(v-1,8)$ with cardinality 
  $s_8(2,1,2)=32$. As an example, the $S_8(2,1)$-adic expansion of $1049$ is given by
  $$
    1049=1\cdot 73+4\cdot 36+26\cdot 32
  $$  
  and the $S_8(2,1)$-adic expansion of $195$ is given by
  $$
    195=3\cdot 73+2\cdot 36-3\cdot 32.
  $$ 
  In the first case the leading coefficient is $26$ and a $32$-divisible multiset of points of cardinality $1049$ is given e.g.\ by 
  $\chi_{E_1}+\sum_{i=1}^4 4\cdot \chi_{L_i}+\sum_{i=1}^{26} 32\cdot \chi_{P_i}$ or by $\sum_{i=1}^5 \chi_{E_i}+\sum_{i=1}^{11} 4\cdot \chi_{L_i}+\sum_{i=1}^{9} 32\cdot \chi_{P_i}$, 
  where the $E_i$ are arbitrary planes, the $L_i$ are arbitrary lines, and the $P_i$ are arbitrary points. In the second case the leading coefficient is $-3$ and the subsequent theorem 
  tells us that no $32$-divisible multiset of points of cardinality $195$ exists in $\PG(v-1,8)$, which also implies that $195=73c_0+36c_1+32c_2$ does not have a solution  
  $\left(c_0,c_1,c_2\right)\in\mathbb{N}^3$. 
\end{example}

Based on the $S_q(a,b)$-adic expansion we can state our main theorem.
\begin{theorem}
  \label{main_thm}
  Let $q=p^e$, $n \in \Z$,and $a,b\in \N$ with $a\ge 1$, $b\le e-1$. The following statements are equivalent:
  \begin{itemize}
    \item[(i)] There exists a $p^{ae-b}$-divisible linear code of effective length $n$ over $\F_q$.
    \item[(ii)] The leading coefficient $c_a$ of the $S_q(a,b)$-adic expansion of $n$ is non-negative.
  \end{itemize}
\end{theorem}  

First we will show the implication (ii) $\Rightarrow$ (i):
\begin{lemma}
  \label{lemma_implication_easy}
  Let $q=p^e$, $n \in \Z$,and $a,b\in \N$ with $a\ge 1$, $b\le e-1$. If the leading coefficient $c_a$ of the $S_q(a,b)$-adic expansion of $n$ is non-negative, then 
  there exists a $p^{ae-b}$-divisible linear code of effective length $n$ over $\F_q$.
\end{lemma}
\begin{proof}
  Let $n =\sum_{i=0}^a c_i \cdot s_q(a,b,i)$ be the $S_q(a,b)$-adic expansion of $n$ and $\Delta:=p^{ae-b}$. By assumption on the non-negativity of the leading coefficient 
  $c_a$ and the definition of the $S_q(a,b)$-adic expansion of $n$ we have $c_i\in\mathbb{N}$ for all $0\le i\le a$. Next we will construct a $\Delta$-divisible multiset 
  $\cM$ of points in $\PG(v-1,q)$ of cardinality $n$. To this end, let $U_i$ be $i$-dimensional subspaces for $1\le i\le a+1$ (assuming that the ambient dimension $v$ is 
  sufficiently large). With this $\cM_0:=\chi_{U_{a+1}}$ is a $q^a$-divisible multiset of points with cardinality $s_q(a,b,0)=[a+1]_q$. Since $\Delta$ divides $q^a$, $\cM_0$ 
  is also $\Delta$-divisible. For $1\le i\le a$ we set $\cM_i:= p^{ie-b} \cdot\chi_{U_{a-i+1}}$, so that $\cM_i$ is $\Delta$-divisible with cardinality $s_q(a,b,i)$. With this, 
  we set $\cM:=\sum_{i=0}^a c_i\cdot \cM_i$, so that $\cM$ is $\Delta$-divisible with cardinality $n$. The corresponding linear code $C_{\cM}$ over $\F_q$ has effective length $n$ 
  and is $\Delta$-divisible.    
\end{proof}   

\begin{lemma}(\cite{sylvester1882subvariants}) 
  \label{lemma_frobenius_two_summands}
  Let $a_1,a_2$ be two positive coprime integers. The largest integer that cannot be written as a non-negative integer linear combination $c_1a_1+c_2a_2$, where $c_1,c_2\in\N$, 
  is given by $g(a_1,a_2):=a_1a_2-a_1-a_2$.
\end{lemma}

\begin{lemma}
  \label{lemma_aux}
  Let $q=p^e$ and $0\le b\le e-1$ be an integer. If $\cM$ is a $p^{e-b}$-divisible multiset of points in $\PG(v-1,q)$ of cardinality $n$, then there exist non-negative integers $s,t$ 
  such that $n=s\cdot (q+1)+t\cdot \Delta$, where $\Delta:=p^{e-b}$.
\end{lemma} 
\begin{proof}
  W.l.o.g.\ we assume $n\ge 1$. Let $k$ be the dimension of the span of $\cM$, i.e.\ the span of the points with positive multiplicity, and w.l.o.g.\ we assume $v=k$. If $k=1$, then we have $\cM(P) \equiv 0 \pmod \Delta$ for the unique 
  point in $\PG(0,q)$ and there exists a non-negative integer $t$ such that $n=t\cdot \Delta$. If $k=2$, then $\PG(1,q)$ consists of $q+1$ pairwise different points $P_0,\dots,P_q$ 
  and we have $\cM(P_i)\equiv n\pmod \Delta$ for all $0\le i\le q$. Now let $\cM'$ arise from $\cM$ by decreasing the points multiplicities by $\Delta$ till we have $\cM(P_i)=s$ for all $0\le i\le q$ for some 
  integer $0\le s<\Delta$. Here $t$ is given by $\big(n-s\cdot (q+1)\big)/\Delta$.
  
  Since $q+1$ and $\Delta$ are coprime, the largest integer that cannot be written as $s=s(q+1)+t\Delta$ for non-negative integers $s,t$ is given by 
  $$
    (q+1)\Delta-(q+1)-\Delta\le q^2+q -(q+1)-\Delta<q^2,
  $$ 
  see Lemma~\ref{lemma_frobenius_two_summands}. 
  So, we can assume $n<q^2$ and $k\ge 3$ in the following. Since $k\ge 3$ there are at least $[3]_q=q^2+q+1>q^2$ points and there 
  exists a point $P$ with multiplicity zero. Let $S$ be a subspace attaining the maximum possible dimension $l$ satisfying $\cM(S)=0$. Clearly we have $l\ge 1$. If $l<k-2$ then consider 
  the $[k-l]_q\ge[3]_q=q^2+q+1>q^2$ $(l+1)$ dimensional subspaces $S'\ge S$. Since $\cM(S')>0$ and these spaces pairwise intersect in $S$ we have $n=\#\cM\ge q^2+q+1 > q^2$ -- 
  contradiction. So, let $S$ be a $(k-2)$-dimensional subspace with $\cM(S)=0$ and consider the $q+1$ hyperplanes $H_0,\dots,H_q$ that contain $S$. Since $\cM(H_i)\equiv n\pmod \Delta$
  for all $0\le i\le q$, there exists an integer $0\le s<\Delta$ such that $\cM(H_i)\equiv s\pmod \Delta$. Since the hyperplanes $H_i$ pairwise intersect in $S$ and $\cM(S)=0$, we have 
  $n=\sum_{i=0}^q \cM(H_i)$, so that $n\ge s\cdot (q+1)$ and $n\equiv s\pmod \Delta$ (using the fact that $\Delta$ divides $q$). Thus, we can set $t=(n-(q+1)s)/\Delta$.
\end{proof}

\begin{lemma}
  \label{lemma_induction_start}
  Theorem~\ref{main_thm} is true for $a=1$.
\end{lemma}
\begin{proof}
  Due to Lemma~\ref{lemma_implication_easy}, it suffices to show the implication (i) $\Rightarrow$ (ii). From Lemma~\ref{lemma_aux} we conclude the existence of $s',t'\in \N$ with 
  $n=s'(q+1)+t'\Delta$. Write $s'=s+x\Delta$ for $s,x\in\N$ with $s<\Delta$ and set $t:=t'+x(q+1)\ge 0$. 
\end{proof}  

\begin{lemma}
  \label{lemma_small_hyperplane}
  (E,g,\ \cite[Lemma 5]{kiermaier2020lengths}) Let $\cM$ be a non-empty multiset of points in $\PG(v-1,q)$, then there exists a hyperplane $H$ with $\cM(H)<\tfrac{\#\cM}{q}$.
\end{lemma}

\begin{lemma}
  \label{lemma_divisibility_hyperplane}
  (\cite[Lemma 4]{kiermaier2020lengths}) Let $\cM$ be a $\Delta$-divisible multiset of points in $\PG(v-1,q)$ and $H$ be an arbitrary hyperplane. If $q$ divides $\Delta$, then 
  the restriction $\cM|_H$ of $\cM$ to $H$ is $\Delta/q$-divisible. 
\end{lemma}  
  
\bigskip

\noindent
\textit{Proof of Theorem~\ref{main_thm}.}
Due to Lemma~\ref{lemma_implication_easy}, it suffices to show the implication (i) $\Rightarrow$ (ii). Using Lemma~\ref{lemma_induction_start} we can assume $a>1$ and prove by induction on $a$.

So, let $n=\sum_{i=0}^a c_i s_q(a,b,i)$ be the $S_q(a,b)$-adic expansion of $n$ and $\sigma=p^bc_0+\sum_{i=1}^a c_i$ be its cross sum. Let $H$ be a hyperplane and set $m:=\cM(H)$. 
Since $\cM$ is $\Delta$-divisible for $\Delta:=p^{ae-b}$ there exists a non-negative integer $\tau$ with $n-m=\tau\Delta$. 
We compute
\begin{eqnarray}
  m &=& n-\tau\Delta = c_0 s_q(a,b,0) \,+\,\sum_{i=1}^{a-1} c_i s_q(a,b,i) \,+\, c_a s_q(a,b,a) -\tau \Delta \notag\\
    &=& c_0 s_q(a-1,b,0) +c_0\cdot q^a+\sum_{i=1}^{a-1} c_i \left(s_q(a-1,b,i)+\Delta\right) + c_a\Delta -\tau\Delta \notag\\ 
    &=& \sum_{i=0}^{a-1} c_i s_q(a-1,b,i) +\big(\sigma-\tau\big)\Delta \label{eq_line_3}\\ 
    &=& \sum_{i=0}^{a-1} c_i s_q(a-1,b,i) +\big(c_{a-1}+q(\sigma-\tau)\big)\Delta/q \label{eq_line_4}   
\end{eqnarray}   
By construction, $\cM|_H$ is $\Delta/q$-divisible, see Lemma~\ref{lemma_divisibility_hyperplane}, with cardinality $m$ and Equation~(\ref{eq_line_4}) gives the $S_q(a-1,b)$-adic expansion of $m$. Hence, 
by induction we get $c_{a-1}+q(\sigma-\tau)\ge 0$. So $q(\sigma-\tau) \ge -c_{a-1}>-q$, implying $\sigma-\tau>-1$ and thus $\sigma\ge \tau$.

By Lemma~\ref{lemma_small_hyperplane} we may choose $H$ such that $m<\tfrac{n}{q}$. Thus using the expression in Equation~(\ref{eq_line_3}) for $m$ we compute
\begin{eqnarray*}
  0 &<& n-qm = \sum_{i=0}^a c_i s_q(a,b,i) - \sum_{i=0}^{a-1} c_i s_q(a-1,b,i)q \,-\, \big(\sigma-\tau\big)\Delta q \\ 
  &=& c_0 +\sum_{i=1}^{a-1} p^{e-b}q^{i-1}c_i  +c_a \Delta -\, \big(\sigma-\tau\big)\Delta q \\ 
  &\le& p^{e-b}-1 + p^{e-b}(q-1)\sum_{i=0}^{a-2}q^i +c_a\Delta \\ 
  &<& p^{e-b}+ p^{e-b}\left(q^{a-1}-1\right) +c_a\Delta \\ 
  &=& (1+c_a)\Delta,
\end{eqnarray*}  
which implies $c_a\ge 0$.\hfill{$\square$}

\medskip

\begin{example}
  \label{ex_old_1}
  For each positive integer $n$ that is either even or at least $5$ a $2$-divisible code of effective length $n$ exists over $\F_4$. For the constructive part we can consider a $2$-fold point, a 
  line, and combinations thereof. For the other direction we can easily check that the leading coefficient of the $S_4(1,1)$-adic expansion of $1$ as well as of $3$ is negative, so that we 
  can apply Theorem~\ref{main_thm}. 
\end{example}

\begin{example}
  \label{ex_old_2}
  For each positive integer $n$ that is not contained in 
  $$
  \{2,4,6,12,14,22\}\cup\{1,3,5,7,9,11,13,15,17,19,23,25,27,33,35,43\}
  $$
  an $8$-divisible linear code of effective length $n$ exists over $\F_4$. Note that an $8$-fold point, a $2$-fold line and a plane are $8$-divisible of cardinalities $8$, $10$, and $21$. 
  The mentioned positive integers are the only ones that cannot be expressed as non-negative integer linear combinations of $8$, $10$, and $21$.
\end{example}

\begin{example}
  \label{ex_old_3}
  For each positive integer $n$ that is either even or at least $9$ a $2$-divisible code of effective length $n$ exists over $\F_8$. For the constructive part we can consider a $2$-fold point, a 
  line, and combinations thereof. For the other direction we can easily check that the leading coefficient of the $S_8(1,2)$-adic expansion of $n\in\{1,3,5,7\}$ is negative, so that we 
  can apply Theorem~\ref{main_thm}.   
\end{example}

We can also consider the dimension $k$ of the span of a $\Delta$-divisible multiset of points $\cM$. If $k=1$, then $\cM$ clearly is a $\lambda$-fold point where $\Delta$ 
divides $\lambda$. Also the case $k=2$ can be easily classified.
\begin{lemma}
  \label{lemma_dimension_2}
  Let $\cM$ be a $\Delta$-divisible multiset of points in $\PG(1,q)$. Then, there exist $l$, possibly equal, points $P_1,\dots,P_l$ such that $\cM=\sum_{i=1}^l \Delta\chi_{P_i}
  +s\chi_L$, where $L$ is the line forming the ambient space and $s=\left(\#\cM-l\Delta\right)/(q+1)\in\N$. Moreover, $\Delta$ divides $qs$.
\end{lemma}
\begin{proof}
  Let $\cM'$ arise from $\cM$ by recursively removing points of multiplicity $\Delta$ and let $l$ be the number of removed points. So, we have 
  $\#\cM'=\#\cM-l\Delta$, the maximum point multiplicity of $\cM'$ is at most $\Delta-1$, and $\cM'$ is also $\Delta$-divisible. From $\Delta$-divisibility we conclude 
  $\cM'(P)\equiv \#\cM'\pmod\Delta$ for every point $P\le L$. Since the maximum point multiplicity of $\cM'$ is at most $\Delta-1$ there exists a non-negative integer $s$
  with $\cM'(P)=s$ for all points $P$, i.e., $\cM'=s\chi_L$. Counting points gives $s=\#\cM'/(q+1)=\left(\#\cM-l\Delta\right)/(q+1)$. Since $\cM'$ is $\Delta$-divisible 
  and $\chi_L$ is $q$-divisible we conclude that $\Delta$ divides $qs$ (including the case $s=0$).
\end{proof}

As we have mentioned the Frobenius coin problem the formulate the result for the largest possible effective length $n$ such that no $p^{ae-b}$-divisible linear 
code over $\F_{p^e}$ exists in this vein:
\begin{proposition}
  Let $q=p^e$ and $a,b\in \N$ with $a\ge 1$, $b\le e-1$. The largest integer that cannot be written as a non-negative integer linear combination 
  $\sum_{i=0}^a c_is_q(a,b,i)$, where $c_0,c_1,\dots,c_a\in\N$, is given by 
  $$
    g\big(s_q(a,b,0),\dots,s_q(a,b,0)\big):=a\cdot p^{(a+1)e-b}-\frac{q^{a+1}-1}{q-1}.
  $$  
\end{proposition}
\begin{proof}
  Given the definition of the $S_q(a,b)$-adic expansion of an integer $n$ we conclude that the largest integer with a negative leading coefficient is given by
  \begin{eqnarray*}
    && \left(p^{e-b}-1\right)\cdot s_q(a,b,0)\,+\,\sum_{i=1}^{a-1} (q-1)\cdot s_q(a,b,i)\,+\,(-1)\cdot s_q(a,b,a) \\ 
    &=& \left(p^{e-b}-1\right)\cdot \frac{q^{a+1}-1}{q-1}\,+\, \sum_{1=1}^{a-1} p^{ie-b} \cdot \left(q^{a-i+1}-1\right) \,-\, p^{ae-b} \\ 
    &=& \left(p^{e-b}-1\right)\cdot \frac{q^{a+1}-1}{q-1}\,+\, (a-1)\cdot p^{(a+1)e-b}-p^{e-b}\cdot \sum_{i=0}^{a-2} q^i\,-\, p^{ae-b} \\ 
    &=& \left(p^{e-b}-1\right)\cdot \frac{q^{a+1}-1}{q-1} - p^{e-b}\cdot \frac{q^{a-1}-1}{q-1}+(a-1)\cdot p^{(a+1)e-b}-p^{ae-b} \\ 
    &=& p^{e-b}\cdot q^{a-1}\cdot (q+1) -\frac{q^{a+1}-1}{q-1} +(a-1)\cdot p^{(a+1)e-b}-p^{ae-b} \\ 
    &=& a\cdot p^{(a+1)e-b}-\frac{q^{a+1}-1}{q-1}.
  \end{eqnarray*}
  With this, the stated result is implied by Theorem~\ref{main_thm}.
\end{proof}
In Examples (\ref{ex_old_1})-(\ref{ex_old_3}) the corresponding numbers $g\big(s_q(a,b,0),\dots,s_q(a,b,0)\big)$ are given by $3$, $43$, and $7$.   

\section{Projective divisible codes}
\label{sec_projective}
In some applications, e.g.\ for upper bounds for partial spreads, see e.g.~\cite{honold2018partial}, the maximum point multiplicity of the multisets of points has to be $1$, i.e., we 
indeed have sets of points and the corresponding linear codes have to be projective. The possible effective lengths of $q^r$-divisible projective codes are very far from being characterized 
and only partial results are known. Here there a papers treating just one length, see \cite{kurz2020no}. The characterization problem is again finite since for every 
$(u+1)$-dimensional space $U$ and every $(u+2)$-dimensional space $U'\ge U$ we have that $\chi_U$ is a $q^u$-divisible set of cardinality $[u+1]_q$ and $\chi_{U'}-\chi_U$ (i.e., the 
characteristic function of an affine subspace) is a $q^u$-divisible set of cardinality $q^u$, so that we can apply Lemma~\ref{lemma_frobenius_two_summands} since $\gcd\!\left([u+1]_q,q^u\right)=1$.
For a recent survey on the possible lengths of $q^r$-divisible projective codes for integers $r$ we refer to \cite[Section 7]{kurz2021divisible}. First few preliminary results for the case 
of restricted column multiplicities larger than $1$ can be found in \cite{korner2023lengths}.

For non-prime field sizes $q$ Baer subspaces and the like give another construction of $\Delta$-divisible point sets:
\begin{lemma}
  Let $q=p^e$ and $1\le f<e$. Then the set of points of an $u$-space $U$ over $\operatorname{GF}(p^e)$ that is also contained in the subfield $\operatorname{GF}(p^f)$ 
  is $p^{uf-e}$ divisible with cardinality $\tfrac{p^{fu}-1}{p^f-1}$ for all $u\in\N_{\ge 3}$. 
\end{lemma}
Note that the assumption $u\ge 3$ is necessary since the hyperplanes of Baer lines are points and there is nothing like a {\lq\lq}Baer point{\rq\rq}, i.e., multiplicities $1$ and
$0$ both occur. Hyperplanes of Baer planes are of course Baer lines and so on.

\begin{example}
  \label{ex_projective_q_4}
  Over $\F_4$ a line gives a $2$-divisible projective code of effective length $5$ and a Baer plane gives a $2$-divisible projective code of effective length $7$. The largest 
  integer that cannot be written as a sum of $5$s and $7$s is $5\cdot 7-5-7=23$, see Lemma~\ref{lemma_frobenius_two_summands}. The linear code corresponding to an affine 
  plane is $4$-divisible and has effective length $25$. 
\end{example}

Family TF1 from \cite{calderbank1986geometry} spelled out in geometrical terms:
\begin{lemma}
  For each integer $e\ge 1$ let $q:=2^e$ and $\cM$ by a hyperoval in $\PG(3-1,q)$. Then, we have $\#\cM=q+2$, $\cM$ is $2$-divisible, and all points have multiplicity at most $1$.
\end{lemma}

Taking lines through a common point without the intersection point yields:
\begin{lemma}
  \label{lemma_two_lines}
  Let $q=p^f$ and $1<e<f$ be an integer. With this let $P$ be a point and $L_1,\dots,L_{p^e}$ be pairwise different lines through $P$ in $\PG(2,q)$. The point set $\cM$ is 
  defined by $\cM(Q)=1$ iff $Q\neq P$ and there exists an index $1\le i\le p^e$ with $Q\le L_i$ and $\cM(Q)=0$ otherwise. Then, we have $\#\cM=p^eq$ and $\cM$ is $p^e$-divisible. 
\end{lemma}

\begin{proposition}
  \label{prop_2_div_q_4_projective}
  A $2$-divisible projective code of effective lengths $n\ge 1$ over $\F_4$ exists iff $n\ge 5$. 
\end{proposition} 
\begin{proof}
  As mentioned in Example~\ref{ex_projective_q_4} $n=5$ and $n=7$ can be attained. Further examples are given by families TF1 and RT3 in \cite{calderbank1986geometry} for 
  $n=6$ and $n=9$, respectively. For $n=8$ and example is given by Lemma~\ref{lemma_two_lines}.  
  It can be easily checked that all $n\ge 5$ can be written as a non-negative integer linear combination of the numbers $5$, $6$, $7$, $8$, and $9$. The non-existence 
  for $n\in\{1,2,3,4\}$ can e.g.\ be checked by exhaustive enumeration (since the maximum possible dimension is $4$ for these lengths).
\end{proof}
We remark that it is also not too hard to give purely theoretical non-existence proofs for $n\in\{1,2,3,4\}$. As observed in Example~\ref{ex_old_1} we can use Theorem~\ref{main_thm} 
to exclude $n\in\{1,3\}$. The non-existence for $n=2$ follows e.g.\ from the classification result in Lemma~\ref{lemma_dimension_2}. 

In order to get e.g.\ the full picture of the possible effective lengths of $2$-divisible projective codes over $\F_8$ we may simply enumerate all such codes using a computer program. To this end 
we have used \texttt{LinCode} \cite{bouyukliev2021computer}, see Table~\ref{table_2_div_q_8_projective}. For field sizes $q=p^e$ we represent the field elements by polynomials over $\F_p$ modulo an irreducible polynomial $f$ of degree $e$. 
Here we use the Conway polynomials $f(\alpha)=\alpha^2+\alpha+1$, $f(\alpha)=\alpha^3+\alpha+1$, and $f(\alpha)=\alpha^2+2\alpha+2$ for $q=4$, $q=8$, and $q=9$, respectively. For an even more 
compact representation we replace $\sum_{i=0}^{e-1} c_i\alpha^i$ by the integer $\sum_{i=0}^{e-1}c_ip^i$.     

\begin{table}[htp]
  \begin{center}
    \begin{tabular}{lrrrrrrrrrrr}
      \hline
      n  & 9 & 10 & 12 & 13 & 14 & 15 & 16 & 17 & 18 & 19 & 20 \\
      \# & 1 &  1 &  1 &  1 &  1 &  3 &  7 &  8 & 20 & 35 & 91 \\  
      \hline
    \end{tabular}
    \caption{Number of non-isomorphic $2$-divisible projective $[n,3]_8$-codes.}
    \label{table_2_div_q_8_projective}  
  \end{center}
\end{table} 

Whenever we have a construction for a $\Delta$-divisible projective code over $\F_q$ of length $n$ such that we do not know such codes with smaller lengths $n_1,n_2$ satisfying $n=n_1+n_2$, 
we speak of a {\lq\lq}base example{\rq\rq}. So, for $\Delta=2$ and $q=4$ we have stated base examples for $n\in\{5,6,7,8,9\}$. For $\Delta=2$ and $q=8$ base examples are given by a line for 
$n=9$, a hyperoval for $n=10$, and one example for $n=16$ constructed via Lemma~\ref{lemma_two_lines}. Besides that we have computationally found the following base examples: 
\begin{center}
\begin{tabular}{ccccc}
  $n=12$: & $n=13$: & $n=14$: & $n=15$: & $n=17$:\\ 
    $\!\!\!\!\begin{pmatrix}
                  111111111100\\
                  001234567010\\
                  136547277001
                \end{pmatrix}\!\!\!\!$
 &  $\!\!\!\!\begin{pmatrix}
                  1111111110100\\
                  0225555661010\\
                  3370237236001
                \end{pmatrix}\!\!\!\!$
 & $\!\!\!\!\begin{pmatrix}
                  11111111111100\\
                  00111234567010\\
                  26124752344001
                \end{pmatrix}\!\!\!\!$
& $\!\!\!\!\begin{pmatrix}
                  111111111110100\\
                  011223366771010\\
                  656675667061001
                \end{pmatrix}\!\!\!\!$
& $\!\!\!\!\begin{pmatrix}
                  11111111111110100\\
                  00022335566771010\\
                  23737032634461001
                \end{pmatrix} \!\!\!\!$                                                             
\end{tabular}
\end{center}
We also found $4$-dimensional $2$-divisible points over $\F_8$ of cardinalities $14$, $16$, $17$, and $18$. In Proposition~\ref{prop_2_div_q_8_projective} we 
fully characterize the possible lengths of $2$-divisible projective codes over $\F_8$.

In the following we state partial results for further divisibility constants $\Delta$ and field sizes $q$.
{\lq\lq}Base examples{\rq\rq} for $4$-divisible projective codes over $\F_8$ 
are given by a line for $n=9$, a two-weight code for $n=28$ \cite[ConstructionTF2]{calderbank1986geometry}, and Lemma~\ref{lemma_two_lines}\ gives an example for $n=32$.  


{\lq\lq}Base examples{\rq\rq} for $2$-divisible projective codes over $\F_{16}$ 
are given by a line for $n=17$, a hyperoval for $n=18$, two-weight codes for $n\in\{21,52,65\}$ \cite{calderbank1986geometry}, and Lemma~\ref{lemma_two_lines}\ gives an example for $n=32$.   


{\lq\lq}Base examples{\rq\rq} for $4$-divisible projective codes over $\F_{16}$
are given by a line for $n=17$, two-weight codes for $n\in\{21,52,65\}$ \cite{calderbank1986geometry}, and Lemma~\ref{lemma_two_lines}\ gives an example for $n=64$.


{\lq\lq}Base examples{\rq\rq} for $8$-divisible projective codes over $\F_{16}$
are given by a line for $n=17$, two-weight codes for $n\in\{120,153,257\}$ \cite{calderbank1986geometry}, and Lemma~\ref{lemma_two_lines}\ gives an example for $n=128$.


{\lq\lq}Base examples{\rq\rq} for $3$-divisible projective codes over $\F_{9}$
are given by a line for $n=10$, a Baer plane for $n=13$, a two-weight code for $n=28$ \cite[Construction RT4]{calderbank1986geometry}, \cite[Example 4.4]{ball2007linear} 
for $n=24$, and the following two codes found by computer enumerations:   
\begin{center}
\begin{tabular}{cc}
  $n=27$: & $n=31$: \\ 
  $\begin{pmatrix}
                111111111111111111111110100\\
                000011223344556666677881010\\
                125628242438071345624474001 
                \end{pmatrix}$
  & $\begin{pmatrix}
                1111111111111111111111111100100\\
                0000011122233344455577788811010\\
                2467802514724612605827812514001     
                \end{pmatrix}$
\end{tabular}
\end{center}

Note that there are no $3$-divisible multisets of points over $\F_9$ with a cardinality in $\{1,2,4,5,7,8,11\}$.

\begin{theorem} (\cite[Theorem 2.1 and Theorem 3.3]{ball2007linear})
  Let $1 <r<q =p^h$ and $\mathcal{S}$ be an $r$-divisible set of points in $\PG(k-1,q)$ whose cardinality is divisible by $r$. Then, we have $|\mathcal{S}|\ge (r-1)q+(p-1)r$
\end{theorem}
If $q$ is even, than the maximal arcs in $\PG(2,q)$ attain the bound of the theorem and they indeed exist for all possible $r$ \cite{denniston1969some}.  
In \cite[Example 4.4]{ball2007linear}) an $r$-divisible set of points in $\PG(2,r^t)$ of cardinality $r^{t+1}-r^{t-1}-r^{t-1}-\dots-r$ was stated. For $t=2$ this yields 
cardinality $r^3-r$.


\bigskip

In the remaining part of this section we consider $2$-divisible (multi-)sets of points over $\F_q$.
\begin{lemma}
  Let $\cM$ be a $2$-divisible multiset of points in $\PG(v-1,q)$. If $q\equiv 1\pmod 2$, then there exists a multiset of points $\cM'$ in $\PG(v-1,q)$ such that $\cM=2\cM'$, so that
  especially $\#\cM\equiv 0\pmod 2$.
\end{lemma}
\begin{proof}
  Since $2$ does not divide $q$ for $q\equiv 1\pmod 2$ the stated result is implied by \cite{ward1981divisible}.
\end{proof}

\begin{lemma}
  \label{lemma_decomposition_2_div}
  Let $\cM$ be a $2$-divisible multiset of points in $\PG(v-1,q)$. Then there exist $s$ points $P_1,\dots,P_s$ and $t$ $2$-divisible sets of points $B_1,\dots,B_t$ such that
  $\cM=\sum_{i=1}^s 2\cdot \chi_{P_i} \,+\, \sum_{i=1}^t \chi_{B_i}$.  
\end{lemma}  
\begin{proof}
  If $Q$ is a point with multiplicity $\cM(Q)\ge 2$, then $\cM-2\cdot \chi_{Q}$ is also $2$-divisible, so that we can assume that $\cM$ is a $2$-divisible set of points after 
  some points with multiplicity $2$ have been removed.
\end{proof}

A direct specialization of Lemma~\ref{lemma_dimension_2} is:
\begin{lemma}
  Let $\cM$ be a $2$-divisible multiset of points in $\PG(1,q)$ where $q$ is even, then we have $\cM=\sum_{i=1}^s 2\cdot\chi_{P_i}+t\cdot\chi_L$ for some points $P_1,\dots,P_s$ 
  (which may coincide) and the line $L$ forming the ambient space. 
\end{lemma}

We call a multiset $\cM$ in $\PG(v-1,q)$ \emph{spanning} if the points with strictly positive multiplicity span the entire ambient space $\PG(v-1,q)$. For a multiset $\cM$ in $V$ and a point $Q$ 
in $V$ the projection $\cM_Q$ is the multiset of points in $V/Q$ with $\cM_Q(L/Q)=\cM(L)-\cM(Q)$ for each line $L\ge Q$ in $V$. It can be easily verified that if $\cM$ is $\Delta$-divisible, then 
$\cM_Q$ is $\Delta$-divisible with cardinality $\#\cM-\cM(Q)$. The maximum point multiplicity may increase up to a factor of $q$. If $\cM$ is spanning, so is $\cM_Q$. 

\begin{proposition}
  \label{prop_lb_2_div}
  Let $\cS$ be a $2$-divisible set of points in $\PG(v-1,q)$ for even field size $q$. Then we have $\#\cS\ge q+1$. Moreover, if $\#\cS=q+1$, then $\cS$ is the characteristic function of a 
  line and if $\#\cS=q+2$, then $\cS$ is the characteristic function of a hyperoval.
\end{proposition}
\begin{proof}
  First we consider the case $\#\cS\equiv 1\pmod 2$ and denote by $Q$ an arbitrary point with multiplicity $0$. (If there is no point of multiplicity zero then we have $\#\cS=\left(q^v-1\right)/(q-1)
  \ge q+1$ and in the case of equality $\cS$ is the characteristic function of a line.) Since $\chi(Q)\not\equiv 1\equiv \#\cS$ we have $v\ge 3$. If $v=3$ then each of the $q+1$ lines through $Q$ has 
  multiplicity at least $1$, so that $\#\cS \ge q+1$. In the case of equality each line $L'$ that contains a point of multiplicity zero satisfies $\cM(L')=1$, so that the line $L$ spanned by 
  $2\le q+1$ points of multiplicity $1$ contains $q+1$ points of multiplicity $1$, i.e., $\cS=\chi_L$. In the following we assume that $\cS$ is spanning, $v\ge 4$, and $\#\cS\le q+1$. Let $L$ be a line 
  with at least two points of multiplicity one and a point $Q$ with multiplicity $0$. Then, $\cS_Q$ is a $2$-divisible multiset with cardinality $q+1$ and at least one point $P$ of multiplicity at 
  least two. By iteratively projecting through points of multiplicity zero we can assume that the ambient space of $\cS_Q$ is three-dimensional. Setting $\cM=\cS_Q-2\cdot\chi_P$ gives a $2$-divisible
  multiset of points with cardinality $\le q-1$ in $\PG(2,q)$. Iteratively removing double points yields a $2$-divisible set $\cS'$ of points with $\#\cS'\le q-1$ and $\#\cS'\equiv 1\pmod 2$, which 
  is impossible as we have seen before. 
  
  Next we consider the case $\#\cS\equiv 0\pmod 2$. Since $v=2$ would imply $\cS(P)\equiv\#\cS\equiv 0\pmod 2$ and $\#\cS=0$, we can assume $v\ge 3$. If $v=3$, then let $P$ be a point of multiplicity $1$. 
  $2$-divisibility implies that the $q+1$ points through $P$ have multiplicity at least $2$, so that $\#\cS\ge q+2$. In the case of equality each line through $P$ has multiplicity $2$, i.e., the line
  multiplicities are contained in $\{0,2\}$ and $\cS$ is the characteristic function of a hyperoval. If $v\ge 4$ and $\cS$ is spanning, then we consider a point $Q$ of multiplicity $1$, so that 
  the projection $\cS_Q$ through $Q$ is a $2$-divisible multiset of odd cardinality $\#\cS-1$. Thus, the previous part implies $\#\cS\ge q+2$. In the case of equality we have 
  $\#\cS_Q=q+1$ and $\cS_Q$ has to be the characteristic function of a line, which contradicts $v\ge 4$ for $\cS$ spanning.   
\end{proof}

The $2$-divisible sets over $\F_2$ with cardinality at most $14$ have been computationally classified in \cite{ubt_eref40887}. A purely theoretical argumentation for cardinalities up to seven can
e.g.\ be found in \cite{korner2023lengths}. For $q=4$ a $2$-divisible set of $q+3=7$ points exists in $\PG(2,4)$, see family RT1 in \cite{calderbank1986geometry}. With a little bit of effort one 
can show that this is the unique possibility of a $2$-divisible set with cardinality $7$.

\begin{proposition}
  \label{prop_no_q+3}
  For even $q>4$ no $2$-divisible set $\cS$ in $\PG(v-1,q)$ of cardinality $q+3$ exists.
\end{proposition}
\begin{proof}
  Clearly we have $v\ge 3$. First we assume the case $v=3$, so that all lines have odd multiplicity. If the support of $\cS$ contains a line $L$, then 
  $\cS-\chi_L$ would be a $2$-divisible set of cardinality $2$, which is impossible. If there is a line $L$ with multiplicity at least $5$, then $L$ contains 
  a point $Q$ with multiplicity $\cS(Q)=0$ and the projection $\cS_Q$ would be $2$-divisible and contains a point of multiplicity at least $4$. Iteratively removing double 
  points would yield a $2$-divisible set of cardinality at most $q-1$ which is impossible. Thus, all lines have multiplicity $1$ or $3$, so that the standard equations yield a contradiction if 
  $q\neq 4$. 
\end{proof}

Proposition~\ref{prop_lb_2_div}, Proposition~\ref{prop_no_q+3}, Lemma~\ref{lemma_decomposition_2_div}, and our stated base examples yield:
\begin{proposition}
  \label{prop_2_div_q_8_projective} 
  A $2$-divisible set $\cS\neq\emptyset$ of cardinality $n$ over $\F_8$ exists iff and only if $n\in\{9,10\}\cup\N_{\ge 12}$.  
\end{proposition}

A $(q + t,t)$-arc of type $(0, 2,t)$ in $\PG(2, q)$, also called KM-arc, see \cite{korchmaros1990}, is a set $\cS$ of $q + t$ points such that every line meets $\cS$ in either $0$, $2$ or $t$ points.

\begin{proposition}
  Let $\cS$ be a $2$-divisible set in $\PG(v-1,q)$ of cardinality $q+4$. If $q>4$ is even, then $\cS$ is a KM-arc of type $(0,2,4)$.
\end{proposition}
\begin{proof}
  W.l.o.g.\ we assume that $\cS$ is spanning. 
  For an arbitrary point $Q$ of multiplicity $\cS(Q)=1$ consider the projection $\cS_Q$ through $Q$, which is a $2$-divisible multiset $\cM$ of cardinality $q+3$. From 
  Lemma~\ref{lemma_decomposition_2_div}, Proposition~\ref{prop_lb_2_div}, and Proposition~\ref{prop_no_q+3} we conclude 
  $\cM=\chi_L+2\cdot\chi_P$ for some line $L$ and some point $P$. If $P\le L$, then we have $v=3$, so that we are in this situation for all points $Q$ with multiplicity $\cS(Q)=1$. Moreover the line 
  multiplicities are contained in $\{0,2,4\}$, so that the statement holds.
  
  Otherwise we have $P\not\le L$ and this is the case for all points $Q$ with multiplicity $\cS(Q)=1$. Thus, we have $v=4$ and the line multiplicities are contained in $\{0,1,2,3\}$. Moreover, the preimage 
  of $L$ is a hyperplane $H$ of cardinality $q+2$ and the other two points outside of $H$ form a line $L$ meeting $Q$. (Actually, $L$ is the unique $3$-line for $Q$.) Now choose a $2$-line $L'$ in $H$ that 
  does not contain $Q$ and a point $P'\neq Q$ on $L$ with multiplicity $\cS(P')=1$. Then, the hyperplane $H':=\langle P',L'\rangle$ has multiplicity $3$, since all points with positive multiplicity 
  are contained in either $H$ or $L$ -- contradiction.  
\end{proof}



\newcommand{\etalchar}[1]{$^{#1}$}

\end{document}